\documentclass[onefignum,onetabnum]{siamart251216}

\hypersetup{
  colorlinks=true,linkcolor=blue,citecolor=blue,urlcolor=blue,
  pdftitle={Global and Unconditional R-Linear Convergence
of Rank-Compressed Weighted LMSD Sweeps for Strictly Convex Quadratics},
  pdfauthor={Shutai Yang and Ya-xiang Yuan},
  pdfkeywords={limited memory steepest descent, algebraic rank compression, spectral conjugacy, Krylov subspaces, finite termination, homogeneous dynamics, R-linear convergence}
}

\usepackage{graphicx}
\graphicspath{{figures/}}
\usepackage{amssymb,amsfonts,mathtools,bm,mathrsfs}
\usepackage{microtype}
\usepackage{array,booktabs,tabularx}
\usepackage{flafter}
\usepackage{placeins}
\usepackage{enumitem}
\usepackage{algpseudocode}

\allowdisplaybreaks
\numberwithin{equation}{section}

\newsiamthm{assumption}{Assumption}
\newsiamremark{remark}{Remark}

\crefname{theorem}{Theorem}{Theorems}
\crefname{lemma}{Lemma}{Lemmas}
\crefname{proposition}{Proposition}{Propositions}
\crefname{corollary}{Corollary}{Corollaries}
\crefname{assumption}{Assumption}{Assumptions}
\crefname{definition}{Definition}{Definitions}
\crefname{remark}{Remark}{Remarks}
\crefname{equation}{equation}{equations}
\crefname{section}{Section}{Sections}
\crefname{appendix}{Appendix}{Appendices}
\newcolumntype{Y}{>{\raggedright\arraybackslash}X}

\newcommand{\R}{\mathbb R}

\newcommand{\norm}[1]{\left\lVert #1\right\rVert}
\newcommand{\abs}[1]{\left\lvert #1\right\rvert}
\newcommand{\Span}{\operatorname{span}}
\newcommand{\rank}{\operatorname{rank}}

\newcommand{\diag}{\operatorname{diag}}
\newcommand{\Kry}{\mathcal K}
\newcommand{\Ip}{\mathcal I_p}
\newcommand{\Pp}{\mathscr Q_p}
\newcommand{\Om}{\Omega_p}
\newcommand{\eps}{\varepsilon}

\title{Global and Unconditional \texorpdfstring{$R$}{R}-Linear Convergence of Rank-Compressed Weighted LMSD Sweeps for Strictly Convex Quadratics\thanks{The work of Shutai Yang was supported by the National College Students Innovation and Entrepreneurship Training Program (Project No.~202510358091). The work of Ya-xiang Yuan was supported by the National Natural Science Foundation of China (Grant No.~12288201).}}

\headers{Convergence of Rank-Compressed Descent Sweeps}{S. Yang and Y.-X. Yuan}

\author{Shutai Yang\thanks{State Key Laboratory of Scientific and Engineering Computing, Academy of Mathematics and Systems Science, Chinese Academy of Sciences, and University of Chinese Academy of Sciences, Beijing, China, and School of Mathematical Sciences, University of Science and Technology of China, Hefei, China
  (\email{yangshutai26@mails.ucas.ac.cn}).}
  \and Ya-xiang Yuan\thanks{State Key Laboratory of Scientific and Engineering Computing, Academy of Mathematics and Systems Science, Chinese Academy of Sciences, Beijing, China
  (\email{yyx@lsec.cc.ac.cn}).}}

\begin{document}

\maketitle

\begin{abstract}
We study limited memory steepest descent (LMSD) with exact algebraic rank compression for strictly convex quadratic optimization. At each cycle, the method restricts the gradient history to its column space and applies the reciprocals of all Ritz values of the compressed projection in the next sweep. If the block-start gradient has at most $p$ active distinct eigenvalues, the delayed sweep terminates finitely. For nonterminating trajectories, a determinant and Cauchy--Binet formula for the complete-sweep polynomial yields global convergence without a full-rank history or a run-wise normalized-conditioning bound. Consecutive sweep endpoints define a continuous positively homogeneous map on a closed cone of compatible states. Compactness then gives an $R$-linear endpoint estimate, uniform over compatible initial states, with constants depending only on $H$ and $p$; the estimate extends to inner gradients, iterate errors, and objective gaps. Every fixed positive spectral weight $W=\omega(H)$ is linearly conjugate to the standard method. The weighted methods therefore share its decay factor, while the Euclidean-norm prefactor is increased by at most $\sqrt{\kappa_2(W)}$. This includes harmonic-Ritz LMSD, fixed power weights, and the delayed BB1 and BB2 recurrences.
\end{abstract}

\section{Introduction}\label{sec:intro}

Consider
\begin{equation}\label{eq:intro-problem}
  \min_{x\in\R^n} f(x),
  \qquad
  f(x)=\tfrac12x^THx-b^Tx,
  \qquad H=H^T\succ0.
\end{equation}
For a gradient step $x^+=x-\alpha g$, the gradient satisfies
$g^+=(I-\alpha H)g$. Spectral gradient methods exploit this
polynomial action by selecting reciprocal stepsizes from spectral
information about $H$. The two classical Barzilai--Borwein stepsizes were introduced in \cite{BarzilaiBorwein1988}. For strictly convex quadratics, Raydan proved global convergence for the first BB choice in arbitrary dimension and observed that the same argument applies to the second \cite{Raydan1993}. Dai and Liao subsequently established an $R$-linear rate for the first choice and noted the analogous result for the second \cite{DaiLiao2002}. Friedlander et al.\ proved global convergence for a finite-delay family of retarded Rayleigh-quotient steps \cite{Friedlander1999}.  For a systematic spectral interpretation and numerical comparison of BB-type, LMSD, and related gradient steplengths, see \cite{DiSerafinoEtAl2018}; a broader survey is given in \cite{ZouMagoules2022}.

Fletcher's limited memory steepest descent method (LMSD) stores the
pre-step gradients generated in a sweep, extracts Ritz values from
their Krylov span, and applies the reciprocal values in the following
sweep, with the Ritz values ordered from largest to smallest
\cite[Sections~2--3]{Fletcher2012}.  In the appendix of that paper, Fletcher proved
that the basic fixed-length standard-Ritz sweep on a strictly convex quadratic either terminates or has gradients converging to zero; Section~7 introduced the corresponding
harmonic-Ritz variant.

Curtis and Guo subsequently established an $R$-linear rate for any
fixed history length for the standard method and, in their appendix,
for the harmonic variant \cite{CurtisGuo2018}.  Their analysis is
carried out under Assumption~3.4.  In particular, if $G_k=Q_kR_k$,
then $G_k$ is required to have full column rank and there must exist a
constant $\rho\ge1$, uniform in the cycle index $k$ along the run,
such that
\begin{equation}\label{eq:CG-assumption-intro}
  \norm{R_k^{-1}}
  \le \rho\,\norm{g_{k,1}}^{-1}.
\end{equation}
Accordingly, their rate constants depend on a run-wise normalized-conditioning bound.

Fletcher \cite[Section~3]{Fletcher2012} and, later, Curtis and Guo \cite[Remark~2.2]{CurtisGuo2018} noted that loss of rank or severe history ill-conditioning may require fewer stored gradients in a finite-precision implementation.  Di Serafino et al.\ form the projected matrix through a Cholesky factorization of $G^TG$ and discard the oldest stored gradient when that factorization loses numerical definiteness \cite[Section~2.2]{DiSerafinoEtAl2018}.  The analyses of Fletcher and of Curtis and Guo use fixed full-rank histories.  Gu and Du's MLMSD instead selects one safeguarded Ritz candidate from the largest full-rank suffix of the available history \cite{GuDu2021}.  Ferrandi and Hochstenbach study Cholesky, pivoted-QR, and SVD realizations, numerical rank truncation, harmonic variants, and nonlinear extensions \cite{FerrandiHochstenbach2025}.

Alignment-based and cyclic-reuse extensions of LMSD were explored by Zou and Magoul\`es \cite{ZouMagoules2019}. Ferrandi, Hochstenbach, and Kreji\'c developed a target-dependent harmonic framework for selecting BB-type spectral steplengths in gradient methods \cite{FerrandiHochstenbachKrejic2023}. Zhou and Gu classified several spectral constructions through determinant pencils and finite moment realizations and proposed a Hanoi-like method that selects and reuses one admissible candidate before recomputing the reduced model \cite{ZhouGu2026}. The present paper considers a different update rule: every value of the algebraically compressed pencil is applied once in the following sweep, including across changes in the algebraic history rank.

We analyze the exact-arithmetic standard-Ritz complete-sweep dynamics on the algebraic column space of each history. A determinant representation of the complete-sweep polynomial gives pointwise global convergence without the full-rank and normalized-conditioning assumptions in \eqref{eq:CG-assumption-intro}. A compatible two-endpoint formulation then yields an $R$-linear estimate uniform over compatible initial states. Fixed positive spectral weights are handled by the change of variables $W^{1/2}$, which preserves the Hessian and the extracted stepsizes.

The main results are as follows.
\begin{enumerate}[label=(\roman*),leftmargin=2.3em]

\item
We formulate standard-Ritz LMSD with exact algebraic rank compression. If a completed history starts from a gradient with at most $p$ active distinct eigenvalues, the compressed projection recovers the active spectrum and the following delayed sweep terminates finitely. This treats the nominal full-rank and rank-deficient termination strata within the same algorithmic rule.

\item
A three-dimensional example directly disproves the fixed-length Ritz
lower-bound implication used to obtain equation~(36) in Fletcher's appendix.
Thus the published proof is incomplete at that step as written unless an additional orbit-specific nondegeneracy argument is supplied. This motivates our use of complete-sweep multipliers rather than individual Ritz roots.

\item
For nonterminating standard sweeps, a determinant ratio and its Cauchy--Binet expansion represent the complete-sweep multiplier through fixed spectral templates with nonnegative state-dependent weights. This representation yields pointwise global convergence across changes in Krylov rank.

\item
Consecutive sweep endpoints define a continuous positively homogeneous self-map on a closed cone of compatible states, with zero continuation on the finite-termination strata. For fixed $H$ and $p$, compactness upgrades pointwise convergence to an $R$-linear endpoint estimate uniform over compatible initial states. The estimate extends to inner gradients, iterate errors, and objective gaps.

\item
Every fixed positive spectral weight $W=\omega(H)$ is linearly conjugate to the standard endpoint map. Hence all such weights admit the standard decay factor, while the Euclidean-norm prefactor is bounded by the standard prefactor times $\sqrt{\kappa_2(W)}$. This covers harmonic-Ritz LMSD, fixed power weights, and the delayed BB1 and BB2 recurrences.
\end{enumerate}

Numerical diagnostics compare fixed power weights on a repeated-spectrum quadratic and record the normalized history-conditioning quantity along standard and harmonic-Ritz trajectories.

Proposition~\ref{prop:finite-branch} identifies the finite-termination strata, while Proposition~\ref{prop:det-formula} supplies the determinant identity used in Theorem~\ref{thm:master} to prove pointwise global convergence. Continuity of the compatible endpoint map in Theorem~\ref{rl:thm:joint-continuity}, combined with the homogeneous compactness principle in Theorem~\ref{rl:thm:homogeneous-rate}, yields the state-uniform endpoint rate in Theorem~\ref{rl:thm:main-fixed}. For an actual run, the rate is initialized at $z_0=(g_{1,0},g_{2,0})$, the start/end pair of the first completed sweep after the warm-up; arbitrary positive finite warm-up steps enter through this initial state, whereas a bound relative to the original gradient would also depend on the warm-up stepsizes. Proposition~\ref{rl:prop:endpoint-conjugacy} and Corollaries~\ref{rl:cor:weighted-endpoint-rate} and~\ref{rl:cor:global-step-rate} transfer the same decay factor to every fixed positive spectral weight and give the corresponding bounds for inner gradients, iterate errors, and objective gaps.

\section{Rank-compressed LMSD sweeps and spectral conjugacy}\label{sec:setup}

\subsection{Active spectral reduction}

Translate the minimizer to the origin, so that
\begin{equation}\label{eq:quadratic}
  f(x)=\frac12x^THx,
  \qquad g=Hx,
  \qquad H=H^T\succ0.
\end{equation}
A gradient step along the negative-gradient direction satisfies
\begin{equation}\label{eq:grad-update}
  g^+=(I-\alpha H)g.
\end{equation}
Consequently every generated gradient is a polynomial in $H$ applied to the initial gradient.  Write the spectral decomposition by distinct eigenvalues as
\begin{equation}\label{eq:spectral-decomposition}
  H=\sum_{\ell=1}^{s}\mu_\ell P_\ell,
  \qquad 0<\mu_1<\cdots<\mu_s.
\end{equation}

\begin{lemma}\label{lem:active-reduction}
Set $u_\ell=P_\ell g_0$ in \eqref{eq:spectral-decomposition}.  For every generated gradient $g_j$ there is a scalar $c_{j,\ell}$ such that
\begin{equation}\label{eq:active-collinear}
  P_\ell g_j=c_{j,\ell}u_\ell.
\end{equation}
Hence the trajectory lies in
\[
  \mathcal A_0=\Span\{u_\ell:u_\ell\ne0\}.
\]
Using the normalized nonzero vectors $u_\ell$ as a basis of $\mathcal A_0$, the restriction of $H$ is diagonal and the eigenvalues of $H|_{\mathcal A_0}$ are distinct.
\end{lemma}

\begin{proof}
If $g_j=p_j(H)g_0$, then
\[
  P_\ell g_j=P_\ell p_j(H)g_0=p_j(\mu_\ell)P_\ell g_0.
\]
This proves \eqref{eq:active-collinear} and the asserted active-coordinate reduction.
\end{proof}

We may therefore work in active spectral coordinates and assume
\begin{equation}\label{eq:simple-spectrum}
  H=\diag(\lambda_1,\ldots,\lambda_n),
  \qquad 0<\lambda_1<\cdots<\lambda_n,
\end{equation}
where $n$ is the number of active distinct eigenvalues. This is the standard active-coordinate reduction used in quadratic spectral-gradient and LMSD analyses; see, e.g., \cite[Section~2]{Fletcher2012} and \cite[Section~2]{DiSerafinoEtAl2018}.

\subsection{History blocks and the delayed LMSD sweep}

A block of $\ell$ gradient steps with positive finite stepsizes is written
\begin{equation}\label{eq:block-steps}
  g_{j+1}=(I-\alpha_jH)g_j,
  \qquad j=0,\ldots,\ell-1,
\end{equation}
and its pre-step history is
\begin{equation}\label{eq:G-generic}
  G=[g_0,\ldots,g_{\ell-1}].
\end{equation}

\begin{proposition}\label{prop:krylov}
For every block \eqref{eq:block-steps},
\begin{equation}\label{eq:span-krylov}
  \Span(G)=\Kry_\ell(H,g_0)
  :=\Span\{g_0,Hg_0,\ldots,H^{\ell-1}g_0\}.
\end{equation}
Thus the span depends only on the block-start gradient and the block length.
\end{proposition}

\begin{proof}
For each $j$,
\[
  g_j=p_j(H)g_0,
  \qquad
  p_j(t)=\prod_{h=0}^{j-1}(1-\alpha_ht).
\]
The polynomial $p_j$ has degree $j$ and nonzero leading coefficient.  Hence the change-of-basis matrix between $[g_0,Hg_0,\ldots,H^{\ell-1}g_0]$ and $G$ is triangular with nonzero diagonal.
\end{proof}

For the algorithmic indexing, write $x_{k,j}$ and $g_{k,j}=Hx_{k,j}$ for the $j$th pre-step state of block $k$.  If block $k$ has length $\ell_k$, then
\begin{equation}\label{eq:block-indexing}
  G_k=[g_{k,0},\ldots,g_{k,\ell_k-1}],
  \qquad
  x_{k,\ell_k}=x_{k+1,0},
  \qquad
  g_{k,\ell_k}=g_{k+1,0}.
\end{equation}
The warm-up block has $\ell_0=p$ unless it terminates early; block $k+1$ has length $\ell_{k+1}=d_k$.  Given a history $G_k$, let $Q_k$ have orthonormal columns spanning $\Span(G_k)$ and set
\begin{equation}\label{eq:dk}
  d_k=\rank(G_k)\le p.
\end{equation}
The standard method extracts the Ritz values
\begin{equation}\label{eq:standard-pencil}
  Q_k^THQ_kv=\theta v.
\end{equation}
They are positive and are used one block later as reciprocal stepsizes.  We list them in nonincreasing order,
\begin{equation}\label{eq:order}
  \theta_{k,1}\ge\cdots\ge\theta_{k,d_k}>0.
\end{equation}

This is the conventional LMSD ordering used by Fletcher \cite[Section~3]{Fletcher2012}.  The endpoint of a completed sweep is independent of this ordering, although the intermediate iterates are not.

\begin{algorithm}[!htbp]
\caption{Standard LMSD sweep with algebraic rank compression}
\label{alg:standard-sweep}
\begin{algorithmic}[1]
\Require $H\succ0$, an initial point, and nominal history length $p$.
\State Generate $p$ warm-up gradient steps with positive finite stepsizes, storing their pre-step gradients in $G_0$, unless $g=0$ is reached earlier.
\If{$g=0$} \State \Return the minimizer. \EndIf
\For{$k=0,1,2,\ldots$}
  \State Compute an orthonormal basis $Q_k$ of $\Span(G_k)$ and set $d_k=\rank(G_k)$.
  \State Solve \eqref{eq:standard-pencil} and order the values as in \eqref{eq:order}.
  \State Starting from the current gradient, apply the stepsizes $\theta_{k,1}^{-1},\ldots,\theta_{k,d_k}^{-1}$ until either termination or completion of the sweep, storing the pre-step gradients as $G_{k+1}$.
  \If{$g=0$} \State \Return the minimizer. \EndIf
\EndFor
\end{algorithmic}
\end{algorithm}
\FloatBarrier

Let $g_{k,0}$ denote the first gradient in history block $G_k$.  On any nonterminating tail for which $d_k=p$, the block-start recurrence is
\begin{equation}\label{eq:delayed-recurrence}
  g_{k+2,0}=\psi_k(H)g_{k+1,0},
  \qquad
  \psi_k(t)=\prod_{j=1}^{p}\left(1-\frac{t}{\theta_{k,j}}\right).
\end{equation}

\subsection{Fixed spectral weights and conjugacy}\label{sec:specializations}

A \emph{positive spectral weight} is a matrix of the form
\begin{equation}\label{eq:spectral-weight-definition}
  W=\sum_{\ell=1}^{s}w_\ell P_\ell=\omega(H),
  \qquad w_\ell>0.
\end{equation}
Thus $W\succ0$, $W$ commutes with $H$, and $WH\succ0$.  We write
\[
  \kappa_2(W):=\norm{W}_2\norm{W^{-1}}_2
\]
for its spectral condition number.  In the active coordinates of \eqref{eq:simple-spectrum}, the weight is diagonal.  The fixed-weight variant uses
\begin{equation}\label{eq:weighted-pencil}
  Q_k^TWHQ_kv=\theta\,Q_k^TWQ_kv.
\end{equation}

\begin{proposition}\label{prop:spectral-conjugacy}
Let $W=\omega(H)\succ0$, set $S=W^{1/2}$, and define
\[
  \widetilde x_j=Sx_j,
  \qquad
  \widetilde g_j=Sg_j,
  \qquad
  \widetilde G_k=SG_k.
\]
Then $\widetilde g_j=H\widetilde x_j$, and every gradient step is conjugated with the same stepsize:
\begin{equation}\label{eq:conjugate-gradient-update}
  \widetilde g_{j+1}=(I-\alpha_jH)\widetilde g_j.
\end{equation}
For every full-column-rank basis $B$ of a history space,
\begin{equation}\label{eq:conjugate-pencil}
  (SB)^TH(SB)=B^TWHB,
  \qquad
  (SB)^T(SB)=B^TWB.
\end{equation}
Consequently the generalized Ritz values of \eqref{eq:weighted-pencil} on $\Span(G_k)$ are exactly the standard Ritz values of $H$ on $\Span(\widetilde G_k)$.  Moreover,
\[
  \rank(\widetilde G_k)=\rank(G_k),
\]
and $g_j$ and $\widetilde g_j$ have the same active spectral support.  Hence the complete fixed-weight rank-compressed LMSD trajectory is the pullback under $S^{-1}$ of the standard rank-compressed trajectory starting from $\widetilde x_0=Sx_0$, with identical warm-up stepsizes, sweep lengths, extracted values, and finite-termination indices.
\end{proposition}

\begin{proof}
Since $S$ and $H$ are functions of the same symmetric matrix, they commute.  Thus $\widetilde g_j=SHx_j=HSx_j=H\widetilde x_j$, and multiplying the gradient recurrence by $S$ gives \eqref{eq:conjugate-gradient-update}.  The same commutation relation and $S^2=W$ give \eqref{eq:conjugate-pencil}.  Generalized eigenvalues are invariant under a change of basis within a fixed subspace, so the two projected problems have the same values.  The matrix $S$ is invertible, which preserves history ranks; its restriction to every eigenspace of $H$ is multiplication by the positive scalar $\sqrt{w_\ell}$, which preserves active spectral support.  The asserted trajectory correspondence now follows block by block.
\end{proof}

Thus it is enough to prove the convergence theory for the standard method; Section~\ref{rl:sec:rates} transfers the resulting bounds back to every fixed positive spectral weight with an explicit norm-equivalence factor.

The weighted pencil \eqref{eq:weighted-pencil} includes the two principal LMSD variants:
\begin{equation}\label{eq:standard-harmonic-special}
  \begin{array}{rcll}
  W=I
  &:& Q_k^THQ_kv=\theta v,
  &\text{standard Ritz LMSD},\\[0.25em]
  W=H
  &:& Q_k^TH^2Q_kv=\theta Q_k^THQ_kv,
  &\text{harmonic-Ritz LMSD}.
  \end{array}
\end{equation}
The algorithm applies the reciprocal $\alpha=\theta^{-1}$ as the stepsize.  For a one-dimensional history, the reciprocal values in the two cases are the quadratic BB1 and BB2 stepsizes, respectively.  For $p>1$, these are the standard and harmonic LMSD pencils studied in \cite{Fletcher2012,CurtisGuo2018,FerrandiHochstenbach2025}.

More generally, because $H\succ0$,
\begin{equation}\label{eq:power-weight}
  W=H^a,\qquad a\in\R,
\end{equation}
is positive definite for every real exponent.  The projected problem becomes
\begin{equation}\label{eq:power-pencil}
  Q_k^TH^{a+1}Q_kv=\theta\,Q_k^TH^aQ_kv.
\end{equation}
Thus $a=0$ is standard LMSD and $a=1$ is harmonic-Ritz LMSD, while other values of $a$ define a power-weighted family.

\subsection{Support, Krylov rank, and finite termination}

For $u=(u_1,\ldots,u_n)^T$, define
\begin{equation}\label{eq:support-def}
  S(u):=\{i:u_i\ne0\},
  \qquad s(u):=|S(u)|.
\end{equation}

\begin{lemma}\label{lem:krylov-rank}
Under \eqref{eq:simple-spectrum},
\begin{equation}\label{eq:krylov-rank}
  \rank[u,Hu,\ldots,H^{p-1}u]=\min\{p,s(u)\}.
\end{equation}
\end{lemma}

\begin{proof}
Delete the zero rows and factor out the nonzero components of $u$.  The remaining matrix is a rectangular Vandermonde matrix with distinct nodes $\lambda_i$, and its rank is $\min\{p,s(u)\}$.
\end{proof}

Lemma~\ref{lem:krylov-rank} separates three cases:
\[
s(u)>p,\qquad s(u)=p,\qquad s(u)<p.
\]
The first is the regular branch. In the second case the
history has rank $p$ but already spans the active invariant subspace;
in the third its rank is $s(u)<p$. Both latter cases lead to finite
termination by the next proposition. 

\begin{proposition}\label{prop:finite-branch}
Let
\[
  G=[g_0,\ldots,g_{p-1}]
\]
be a completed $p$-step history block with nonzero block-start gradient $g:=g_0$ and block endpoint $h:=g_p$.  Suppose that $d=s(g)\le p$ (hence $1\le d\le p$).  Then
\begin{equation}\label{eq:active-invariant}
  \Span(G)=\Kry_p(H,g)=\Span\{e_i:i\in S(g)\}.
\end{equation}
The eigenvalues of the standard projected matrix on $\Span(G)$ are precisely $\{\lambda_i:i\in S(g)\}$.  The effective standard sweep generated from this history and applied to $h$ annihilates $h$ in at most $d$ steps.
\end{proposition}

\begin{proof}
Proposition~\ref{prop:krylov} gives $\Span(G)=\Kry_p(H,g)$.  By Lemma~\ref{lem:krylov-rank}, this space has dimension $d$.  The rightmost space in \eqref{eq:active-invariant} is a $d$-dimensional $H$-invariant subspace containing $\Kry_p(H,g)$, so \eqref{eq:active-invariant} holds.

On this active invariant subspace, the compressed standard projection is the restriction of $H$.  Therefore its eigenvalues are exactly $\{\lambda_i:i\in S(g)\}$.  The endpoint of the history block has the form
\[
  h=p_p(H)g,
  \qquad
  p_p(t)=\prod_{j=0}^{p-1}(1-\alpha_jt),
\]
so
\begin{equation}\label{eq:finite-support-inclusion}
  S(h)\subseteq S(g).
\end{equation}
Define
\[
  \zeta_g(t):=\prod_{i\in S(g)}\left(1-\frac{t}{\lambda_i}\right).
\]
For each spectral coordinate $r$, either $h_r=0$ or $r\in S(g)$, in which case $\zeta_g(\lambda_r)=0$.  Consequently
\[
  \zeta_g(H)h
  =\prod_{i\in S(g)}\left(I-\lambda_i^{-1}H\right)h
  =0.
\]
Thus the values extracted from the history starting at $g$ annihilate the endpoint $h$ to which the delayed sweep is applied.  The order of the $d$ factors is immaterial for the endpoint, and termination may occur before all $d$ factors have been used.
\end{proof}

\begin{corollary}\label{cor:weighted-finite-branch}
Proposition~\ref{prop:finite-branch} holds unchanged for every fixed positive spectral weight $W=\omega(H)$: the weighted compressed pencil recovers the same active eigenvalues and the effective delayed sweep terminates in at most $d$ steps.
\end{corollary}

\begin{proof}
Apply Proposition~\ref{prop:spectral-conjugacy}. The transformation by $W^{1/2}$ preserves active spectral support and history rank, while the weighted pencil has exactly the standard Ritz values of the transformed history.
\end{proof}

\begin{corollary}\label{cor:full-tail}
On every nonterminating standard run, each completed history has rank $p$ and
each block-start gradient satisfies
\[
s(g_{k,0})>p.
\]
Consequently all blocks after the warm-up have length $p$, and
\eqref{eq:delayed-recurrence} holds throughout the nonterminating
tail.
\end{corollary}

\begin{proof}
Suppose, to the contrary, that $k$ is the first index with
$d_k<p$. Since the run is nonterminating, $G_0$ is a completed
$p$-step warm-up history. If $k>0$, minimality gives
$d_{k-1}=p$, so $G_k$ was also generated by a completed $p$-step
sweep. Hence Proposition~\ref{prop:krylov} and
Lemma~\ref{lem:krylov-rank} give
\[
d_k=\min\{p,s(g_{k,0})\}<p.
\]
Proposition~\ref{prop:finite-branch} then terminates the following
sweep, a contradiction. Thus $d_k=p$ for every $k$.

If some block-start gradient satisfied $s(g_{k,0})=p$, the same
proposition would again terminate the following sweep. Therefore
$s(g_{k,0})>p$ on every nonterminating run.
\end{proof}

\begin{remark}\label{rem:p-vs-n}
If $p\ge n$ in the reduced problem, every nonzero block-start gradient has at most
$p$ active eigenvalues, so Proposition~\ref{prop:finite-branch}
implies finite termination. In particular, the total number of
gradient steps is at most $p+n$.
\end{remark}

Although every finite history on a nonterminating run has rank $p$,
normalized histories may approach lower-support boundary states. This
issue is examined in Section~\ref{sec:rankloss}.

\section{Fixed-length Ritz continuity at Krylov-rank loss}
\label{sec:rankloss}

At a rank-changing limit, convergence of normalized block-start gradients does not by itself control all Ritz values of a projected problem whose prescribed dimension exceeds the limiting Krylov rank. The following example isolates this difficulty for the lower bound used in Fletcher's equation~(36).

\begin{proposition}\label{prop:counterexample}
Let
\begin{equation}\label{eq:counter-H}
  H=\diag(1,2,3),
  \qquad p=2,
  \qquad
  q_\eps:=\frac{(\eps,\eps^2,1)^T}{\sqrt{1+\eps^2+\eps^4}},
  \qquad \eps>0.
\end{equation}
Let $\theta_{\eps,1}\ge\theta_{\eps,2}$ be the two standard Ritz values of $H$ on $\Kry_2(H,q_\eps)$.  Then
\begin{equation}\label{eq:counter-limits}
  q_\eps\to e_3,
  \qquad
  \Kry_2(H,q_\eps)\longrightarrow\Span\{e_1,e_3\},
  \qquad
  (\theta_{\eps,1},\theta_{\eps,2})\to(3,1).
\end{equation}
Here the subspace convergence is in the operator norm of the corresponding orthogonal projectors. Consequently, for all sufficiently small $\eps$,
\begin{equation}\label{eq:fletcher-bound-fails}
  \theta_{\eps,2}<2=\frac23\lambda_3.
\end{equation}
The same pair limit $(3,1)$ holds for the generalized Ritz values associated with every fixed positive spectral weight $W=\omega(H)$.
\end{proposition}

\begin{proof}
Clearly $q_\eps\to e_3$.  Set
\[
  r_\eps:=\frac{(H-3I)q_\eps}{\norm{(H-3I)q_\eps}}
  =\frac{(-2,-\eps,0)^T}{\sqrt{4+\eps^2}}.
\]
Then $r_\eps\in\Kry_2(H,q_\eps)$ and $r_\eps\to-e_1$.  Moreover $q_\eps^Tr_\eps\to0$.  Orthogonalizing $r_\eps$ against $q_\eps$ therefore gives a unit vector $\widehat r_\eps\in\Kry_2(H,q_\eps)$ with $\widehat r_\eps\to-e_1$.  Hence
\[
  Q_\eps:=[q_\eps,\widehat r_\eps]\longrightarrow[e_3,-e_1]
\]
is a convergent orthonormal basis, which proves the subspace limit in \eqref{eq:counter-limits}.  It follows that
\[
  Q_\eps^THQ_\eps\longrightarrow\diag(3,1),
\]
so the ordered Ritz values tend to $(3,1)$ and \eqref{eq:fletcher-bound-fails} follows.

For a fixed positive spectral weight $W=\diag(w_1,w_2,w_3)$,
\[
  Q_\eps^TWHQ_\eps\longrightarrow\diag(3w_3,w_1),
  \qquad
  Q_\eps^TWQ_\eps\longrightarrow\diag(w_3,w_1).
\]
The limiting generalized eigenvalues are again $3$ and $1$.
Equivalently, the orthogonal projectors onto $\Kry_2(H,q_\eps)$ converge in operator norm to the projector onto $\Span\{e_1,e_3\}$.
\end{proof}

\paragraph{Relation to Fletcher's equation~(36)}
In Fletcher's appendix, let $r$ denote the spectral-component induction index (denoted $p$ there) and let $m$ denote the sweep length. Along an infinite subsequence, the first $r-1$ components of the normalized sweep-start gradients tend to zero, so every accumulation point is supported on $\{r,\ldots,n\}$. The proof then invokes continuity of eigenvalues to conclude that the $m$ prelimit Ritz values eventually satisfy
\[
  \theta_{k,\ell}\in(2\lambda_r/3,\lambda_n),
  \qquad \ell=1,\ldots,m,
\]
which is equation~(36) of \cite[Appendix]{Fletcher2012}.

The difficulty is that the limiting $m$-step Krylov problem may have rank below $m$.  Proposition~\ref{prop:counterexample}, with $r=3$ and $m=2$, gives a path for which the normalized block-start gradient tends to $e_3$, whereas the smaller prelimit Ritz value tends to $1<2=(2/3)\lambda_3$.  Hence the normalized-gradient convergence and finite-iteration full-rank information used at this point in the appendix do not by themselves imply equation~(36).  Recovering that conclusion would require an additional orbit-specific constant-rank or quantitative nondegeneracy argument that excludes such rank-changing limiting paths.

Our convergence proof avoids the issue by working with the complete-sweep multiplier rather than an ordered Ritz tuple. The next section derives determinant and Cauchy--Binet representations that remain suitable at rank-changing limits.

\section{Determinant representation of a complete sweep}\label{sec:mainproof}

For each completed block,
\[
  g_{k+1,0}=q_k(H)g_{k,0}
\]
for some polynomial $q_k$, and hence
\[
  S(g_{k+1,0})\subseteq S(g_{k,0}).
\]
The support therefore stabilizes after finitely many blocks. On a
nonterminating run its cardinality exceeds $p$ by
Corollary~\ref{cor:full-tail}. Restricting to the resulting invariant
subspace and relabeling the eigenvalues, we may analyze the tail under
\begin{equation}\label{eq:regular-standing}
  n>p,\qquad
  g_{k,0}^{(i)}\ne0
  \quad(i=1,\ldots,n),\qquad k\ge k_0.
\end{equation}

\subsection{Spectral enclosure}

\begin{lemma}\label{lem:enclosure}
Let $Q\in\R^{n\times d}$ have orthonormal columns.  Then every eigenvalue of the standard projected matrix $Q^THQ$ lies in
\begin{equation}\label{eq:enclosure}
  [\lambda_1,\lambda_n].
\end{equation}
\end{lemma}

\begin{proof}
This follows directly from the min--max principle applied to the restriction of $H$ to $\Span(Q)$.
\end{proof}

\subsection{The complete-sweep multiplier}

On a regular block, write $g_k:=g_{k,0}$ and define
\begin{equation}\label{eq:K-DV}
  K_k=[g_k,Hg_k,\ldots,H^{p-1}g_k]=D_kV,
\end{equation}
where
\begin{equation}\label{eq:DV}
  D_k=\diag(g_k^{(1)},\ldots,g_k^{(n)}),
  \qquad V_{ij}=\lambda_i^{j-1}.
\end{equation}
Let $\Ip$ denote the family of all $p$-element subsets $I\subset\{1,\ldots,n\}$.  For $I\in\Ip$, let $V_I$ be the corresponding $p\times p$ row submatrix of $V$ and set $\Delta_I=\det(V_I)$; under \eqref{eq:simple-spectrum} every $\Delta_I$ is a Vandermonde determinant with distinct nodes and is therefore nonzero.

\begin{proposition}\label{prop:det-formula}
For every $t\in\R$, the complete-sweep multiplier satisfies
\begin{equation}\label{eq:det-ratio}
  \psi_k(t)
  =\frac{\det\!\left(K_k^T(H-tI)K_k\right)}
         {\det\!\left(K_k^THK_k\right)}.
\end{equation}
For a spectral index $r$, define
\begin{equation}\label{eq:beta-phi}
  \beta_I^{(k)}
  =\Delta_I^2\prod_{i\in I}\lambda_i\bigl(g_k^{(i)}\bigr)^2>0,
  \qquad
  \phi_I^{(r)}
  =\prod_{i\in I}\left(1-\frac{\lambda_r}{\lambda_i}\right),
\end{equation}
and
\begin{equation}\label{eq:pi-def}
  \pi_I^{(k)}
  =\frac{\beta_I^{(k)}}{\sum_{J\in\Ip}\beta_J^{(k)}}.
\end{equation}
Then
\begin{equation}\label{eq:central-average}
  \psi_k(\lambda_r)
  =\sum_{I\in\Ip}\pi_I^{(k)}\phi_I^{(r)},
  \qquad
  \pi_I^{(k)}>0,
  \qquad
  \sum_{I\in\Ip}\pi_I^{(k)}=1.
\end{equation}
Thus $\psi_k(\lambda_r)$ is a convex combination of the fixed real numbers $\phi_I^{(r)}$.
\end{proposition}

\begin{proof}
In the regular regime, $K_k$ has full column rank. Since
$H\succ0$, the matrix
$K_k^THK_k$ is positive definite, so the denominator in
\eqref{eq:det-ratio} is nonzero. The generalized eigenvalues are
unchanged when the orthonormal basis $Q_k$ is replaced by the
full-rank basis $K_k$ of the same subspace.  Hence
\[
  \det\!\left(K_k^THK_k-\lambda K_k^TK_k\right)
  =\det(K_k^TK_k)\prod_{j=1}^{p}(\theta_{k,j}-\lambda).
\]
Evaluating at $\lambda=t$ and at $\lambda=0$ and dividing gives \eqref{eq:det-ratio}, because $\prod_j(\theta_{k,j}-t)/\prod_j\theta_{k,j}=\psi_k(t)$.  Specializing to $t=\lambda_r$, substituting $K_k=D_kV$, and applying Cauchy--Binet to the denominator gives
\begin{equation}\label{eq:CB-den}
  \det(K_k^THK_k)=\sum_{I\in\Ip}\beta_I^{(k)}.
\end{equation}
The same algebraic Cauchy--Binet expansion for the numerator gives
\begin{equation}\label{eq:CB-num}
  \det\!\left(K_k^T(H-\lambda_rI)K_k\right)
  =\sum_{I\in\Ip}\beta_I^{(k)}\phi_I^{(r)}.
\end{equation}
Dividing \eqref{eq:CB-num} by \eqref{eq:CB-den} yields \eqref{eq:central-average}.
\end{proof}

For completeness, we record the tail form of the scalar lemma used in Fletcher's appendix \cite{Fletcher2012}.

\begin{lemma}\label{lem:two-step}
Let $k_0\ge0$ and let $\{a_k\}_{k\ge k_0}$ be a nonnegative sequence.  Suppose there are constants $M\ge1$ and $c\in(0,1)$ such that
\begin{equation}\label{eq:M-bound}
  a_{k+2}\le Ma_{k+1}
  \qquad(k\ge k_0).
\end{equation}
Assume further that for every $\eps>0$ there is $k_\eps\ge k_0$ such that
\begin{equation}\label{eq:trigger}
  k\ge k_\eps,
  \quad a_k\ge\eps
  \quad\Longrightarrow\quad
  a_{k+2}\le c a_{k+1}.
\end{equation}
Then $a_k\to0$.
\end{lemma}

\begin{proof}
Fix $\delta>0$.  We first show that there exists $N\ge k_\delta+1$ such that $a_N<\delta$.  Otherwise $a_k\ge\delta$ for every $k\ge k_\delta+1$, and
\eqref{eq:trigger} would give
\[
  a_{k+2}\le c a_{k+1}
  \qquad(k\ge k_\delta+1).
\]
Equivalently, $a_{j+1}\le c a_j$ for every sufficiently large $j$, which forces $a_j\to0$ and contradicts $a_j\ge\delta$.

Fix such an $N$.  We claim that
\[
  a_j<M^2\delta
  \qquad\text{for every }j\ge N.
\]
For $j=N$ this follows from $a_N<\delta$ and $M\ge1$.  Moreover, since $N-1\ge k_0$, \eqref{eq:M-bound} gives
\[
  a_{N+1}\le Ma_N<M\delta\le M^2\delta.
\]

Assume inductively that the claim holds through index $m$, where $m\ge N+1$.  If $a_{m-1}<\delta$, then two applications of
\eqref{eq:M-bound} give
\[
  a_{m+1}\le Ma_m\le M^2a_{m-1}<M^2\delta.
\]
If $a_{m-1}\ge\delta$, then $m-1\ge k_\delta$, and
\eqref{eq:trigger} gives
\[
  a_{m+1}\le c a_m<a_m<M^2\delta.
\]
The induction is complete.  Hence
\[
  \limsup_{k\to\infty}a_k\le M^2\delta.
\]
Letting $\delta\downarrow0$ proves that $a_k\to0$.
\end{proof}

We now combine Proposition~\ref{prop:det-formula} and Lemma~\ref{lem:two-step} in a spectral induction.

\section{Global convergence}\label{sec:global}

\begin{theorem}\label{thm:master}
Consider Algorithm~\ref{alg:standard-sweep} for \eqref{eq:quadratic} with algebraic rank compression.  The method either reaches $g=0$ in finitely many steps or, if it does not terminate,
\begin{equation}\label{eq:within-conv}
  g_{k,0}\longrightarrow0,
  \qquad
  \max_{0\le j\le p}\norm{g_{k,j}}\longrightarrow0,
  \qquad
  \max_{0\le j\le p}\norm{x_{k,j}}\longrightarrow0.
\end{equation}
In the original variables of \eqref{eq:intro-problem}, before the translation and active reduction of Section~\ref{sec:setup}, the conclusion states that every iterate converges to the unique minimizer of $f$.
\end{theorem}

\begin{proof}
Assume that the method does not terminate. By
Corollary~\ref{cor:full-tail}, every post-warm-up block has length
$p$ and every block-start gradient has more than $p$ active
components. We therefore work on the regular tail
\eqref{eq:regular-standing}.

Set
\[
  B:=\max\left\{1,\frac{\lambda_n}{\lambda_1}-1\right\}.
\]
By Lemma~\ref{lem:enclosure}, every extracted value lies in $[\lambda_1,\lambda_n]$.  Therefore, for every $r$ and $k$,
\begin{equation}\label{eq:rough-bound}
  \abs{\psi_k(\lambda_r)}\le B^p.
\end{equation}

We prove componentwise convergence by induction on $r$.  For $r=1$, every set $I$ containing $1$ has $\phi_I^{(1)}=0$.  If $1\notin I$, then
\[
  0<\phi_I^{(1)}
  \le\left(1-\frac{\lambda_1}{\lambda_n}\right)^p
  =:\rho_1<1.
\]
The convex-weight identity \eqref{eq:central-average} gives
\begin{equation}\label{eq:first-contract}
  \abs{\psi_k(\lambda_1)}\le\rho_1,
\end{equation}
so \eqref{eq:delayed-recurrence} implies $g_{k,0}^{(1)}\to0$.

Let $r\ge2$ and assume $g_{k,0}^{(i)}\to0$ for every $i<r$.  We verify the hypotheses of Lemma~\ref{lem:two-step} for the sequence
\[
  a_k=\abs{g_{k,0}^{(r)}},\qquad k\ge k_0,
\]
with $k_0$ the start of the regular tail \eqref{eq:regular-standing}. The recurrence \eqref{eq:delayed-recurrence} gives
\[
  a_{k+2}=\abs{\psi_k(\lambda_r)}\,a_{k+1}.
\]
Thus \eqref{eq:rough-bound} yields the two-step bound \eqref{eq:M-bound} with $M=B^p\ge1$.

It remains to verify the trigger implication \eqref{eq:trigger}.  Fix an arbitrary $\eps>0$ and put
\begin{equation}\label{eq:S-eps}
  S_\eps=\{k\ge k_0:\abs{g_{k,0}^{(r)}}\ge\eps\}.
\end{equation}
If $S_\varepsilon$ is finite, the trigger condition is immediate after its last element. Assume that $S_\varepsilon$ is infinite.  Partition $\Ip$ into
\begin{align}
  A_r&=\{I:r\in I\},\label{eq:partition-A}\\
  B_r&=\{I:I\subset\{r+1,\ldots,n\}\},\label{eq:partition-B}\\
  C_r&=\Ip\setminus(A_r\cup B_r).\label{eq:partition-C}
\end{align}
For $I\in A_r$, $\phi_I^{(r)}=0$.  For $I\in B_r$,
\begin{equation}\label{eq:rho-r}
  0<\phi_I^{(r)}
  \le\left(1-\frac{\lambda_r}{\lambda_n}\right)^p
  =:\rho_r<1,
\end{equation}
with $\rho_r=0$ if $B_r$ is empty.

Every $I\in C_r$ omits $r$ and contains at least one index $i(I)<r$.  Define
\[
  J(I)=(I\setminus\{i(I)\})\cup\{r\}\in A_r.
\]
By \eqref{eq:beta-phi},
\begin{equation}\label{eq:ratio-concrete}
  \frac{\beta_I^{(k)}}{\beta_{J(I)}^{(k)}}
  =c_I
   \frac{(g_{k,0}^{(i(I))})^2}{(g_{k,0}^{(r)})^2},
\end{equation}
where
\[
  c_I
  =\frac{\Delta_I^2}{\Delta_{J(I)}^2}
   \frac{\lambda_{i(I)}}{\lambda_r}>0
\]
is independent of $k$.  Since the denominator in \eqref{eq:pi-def} is at least $\beta_{J(I)}^{(k)}$, for $k\in S_\eps$,
\begin{equation}\label{eq:C-weight-vanishes}
  \pi_I^{(k)}
  \le\frac{\beta_I^{(k)}}{\beta_{J(I)}^{(k)}}
  \le\frac{c_I}{\eps^2}(g_{k,0}^{(i(I))})^2
  \longrightarrow0
  \qquad(k\in S_\eps,\ k\to\infty),
\end{equation}
by the induction hypothesis.
There are only finitely many sets in $C_r$, and every $\phi_I^{(r)}$ is a fixed finite scalar.  For $k\in S_\eps$,
\[
  \abs{\psi_k(\lambda_r)}
  \le \rho_r\sum_{I\in B_r}\pi_I^{(k)}
     +\sum_{I\in C_r}\pi_I^{(k)}\abs{\phi_I^{(r)}}
  \le \rho_r
     +\sum_{I\in C_r}\pi_I^{(k)}\abs{\phi_I^{(r)}}.
\]
Equations \eqref{eq:C-weight-vanishes} and \eqref{eq:central-average} therefore imply
\begin{equation}\label{eq:limsup-psi}
  \limsup_{\substack{k\to\infty\\ k\in S_\eps}}\abs{\psi_k(\lambda_r)}\le\rho_r<1.
\end{equation}
Choose the fixed contraction factor
\[
  c_r:=\frac{1+\rho_r}{2}\in(0,1).
\]
For every $\varepsilon>0$, \eqref{eq:limsup-psi} yields an index
$k_\varepsilon$ such that
\begin{equation}\label{eq:trigger-component}
  \abs{g_{k+2,0}^{(r)}}
  \le c_r\abs{g_{k+1,0}^{(r)}}
  \qquad\text{for every }k\in S_\eps\text{ with }k\ge k_\eps,
\end{equation}
while for $k\notin S_\eps$ the premise of \eqref{eq:trigger} is not met.  The factor $c=c_r$ therefore verifies \eqref{eq:trigger} for every $\eps>0$, with $k_\eps$ depending on $\eps$.  All hypotheses of Lemma~\ref{lem:two-step} are satisfied, and the lemma yields $a_k\to0$, that is, $g_{k,0}^{(r)}\to0$.  The induction is complete, so $g_{k,0}\to0$.

Finally, each individual factor $I-H/\theta$ has spectral norm at most $B$, and a nonterminating block has $p$ steps.  Therefore
\[
  \max_{0\le j\le p}\norm{g_{k,j}}
  \le B^p\norm{g_{k,0}}\longrightarrow0.
\]
Since $x_{k,j}=H^{-1}g_{k,j}$, the iterate convergence in \eqref{eq:within-conv} follows.  Undoing the translation and the active reduction of Section~\ref{sec:setup} restores the statement in the original variables.
\end{proof}

\begin{corollary}\label{cor:weighted-global}
For every fixed positive spectral weight $W=\omega(H)$, the rank-compressed weighted LMSD method defined by \eqref{eq:weighted-pencil} has the same finite-termination alternative and pointwise convergence conclusions as Theorem~\ref{thm:master}.
\end{corollary}

\begin{proof}
Proposition~\ref{prop:spectral-conjugacy} identifies the transformed weighted trajectory with a standard trajectory for the same Hessian.  Apply Theorem~\ref{thm:master} and then use the invertibility of $W^{1/2}$.
\end{proof}

We now formulate the standard sweep dynamics as a first-order map on consecutive endpoints.

\section{Endpoint-state formulation}\label{rl:sec:sweeps}

To write the delayed recurrence as a first-order system, we define the compatible endpoint pairs before spectral reduction.  For an SPD matrix $H$, let
\begin{equation}\label{rl:eq:Pp-H}
  \mathscr Q_p(H)
  :=\left\{
  q(t)=\prod_{j=1}^{p}(1-\alpha_jt):
  \alpha_j\in[\lambda_{\max}(H)^{-1},\lambda_{\min}(H)^{-1}]
  \right\},
\end{equation}
and define
\begin{equation}\label{rl:eq:Omega-H}
  \Omega_p(H)
  :=\{(u,v):v=q(H)u\text{ for some }q\in\mathscr Q_p(H)\}.
\end{equation}
We use the product norm
\begin{equation}\label{rl:eq:product-norm}
  \norm{(u,v)}_\times:=\bigl(\norm{u}^2+\norm{v}^2\bigr)^{1/2}.
\end{equation}

This cone contains every consecutive endpoint pair of a nonterminating sweep and serves as the state space for the rate analysis.

For $(u,v)\in\Omega_p(H)$, compatibility gives $v=q(H)u$ for some
$q\in\mathscr Q_p(H)$. Hence both endpoints and all subsequent
polynomial updates lie in
\[
  E(u)=\Span\{P_\ell u:P_\ell u\ne0\}.
\]

Identifying the normalized nonzero vectors $P_\ell u$ with coordinate vectors reduces $H$ isometrically to a diagonal matrix with distinct eigenvalues, while preserving projected matrices, endpoint updates, and product norms. We therefore work in the coordinates
\eqref{eq:simple-spectrum}, set
\[
  K_p(u)=[u,Hu,\ldots,H^{p-1}u],
\]
retain the support notation $S(u)$ and $s(u)$ from Section~\ref{sec:setup}, and abbreviate
\[
  \Pp=\mathscr Q_p(H),
  \qquad
  \Om=\Omega_p(H).
\]
The cases $s(u)>p$, $s(u)=p$, and $s(u)<p$ correspond respectively to the regular branch, the full-rank termination boundary, and the rank-deficient termination boundary.

If $s(u)\ge p$, Lemma~\ref{lem:krylov-rank} gives $\rank K_p(u)=p$.  The generalized eigenvalues of
\begin{equation}\label{rl:eq:K-pencil}
  K_p(u)^THK_p(u)z
  =\theta K_p(u)^TK_p(u)z
\end{equation}
are therefore well defined; denote them by $\theta_1(u),\ldots,\theta_p(u)$ and set
\begin{equation}\label{rl:eq:Psi}
  \Psi_u(t)
  :=\prod_{j=1}^{p}\left(1-\frac{t}{\theta_j(u)}\right).
\end{equation}
Because $H$ is symmetric positive definite and $K_p(u)$ has full column rank,
\[
  K_p(u)^THK_p(u)\succ0.
\]
Consequently,
\begin{equation}\label{rl:eq:Psi-det}
  \Psi_u(t)
  =\frac{\det\!\left(K_p(u)^T(H-tI)K_p(u)\right)}
         {\det\!\left(K_p(u)^THK_p(u)\right)}
\end{equation}
has a strictly positive denominator.  If $s(u)=p$, the columns of $K_p(u)$ span the active invariant subspace, so the values in \eqref{rl:eq:K-pencil} are exactly $\{\lambda_i:i\in S(u)\}$ and
\begin{equation}\label{rl:eq:s-equals-p-annihilation}
  \Psi_u(H)v=0
  \qquad\text{for every }(u,v)\in\Om\text{ with }s(u)=p.
\end{equation}
Here compatibility gives $S(v)\subseteq S(u)$.  When $s(u)<p$, the $p$-column Krylov matrix is rank deficient; exact compression returns the $s(u)$ active eigenvalues and the following effective sweep terminates.

Define the endpoint output by
\begin{equation}\label{rl:eq:F-def}
  F_p:\Om\longrightarrow\R^n,
  \qquad
F_p(u,v):=
\begin{cases}
  \Psi_u(H)v, & s(u)>p,\\[0.2em]
  0, & s(u)\le p.
\end{cases}
\end{equation}
When $s(u)=p$, the projected values are the active eigenvalues, so
$\Psi_u(H)v=0$.  In particular, $F_p(0,0)=0$.  Define the transition
\begin{equation}\label{rl:eq:T-def}
  T_p:\Om\longrightarrow\R^n\times\R^n,
  \qquad
  T_p(u,v):=\bigl(v,F_p(u,v)\bigr).
\end{equation}
Section~\ref{rl:sec:cone} shows that $T_p$ is a continuous self-map of $\Omega_p(H)$.

\section{A compactness principle for homogeneous dynamics}\label{rl:sec:homogeneous}

Let $X$ be a finite-dimensional normed space. A nonempty closed set $\Omega\subset X$ is a closed cone if $0\in\Omega$ and $t\Omega\subseteq\Omega$ for every $t\ge0$. Related stability results for homogeneous discrete-time systems are discussed in, for example \cite{ShenHu2012}. Here we give a short theorem which directly supplies the uniform $R$-linear estimate needed below.

\begin{theorem}\label{rl:thm:homogeneous-rate}
Let $\Omega\subset X$ be a closed cone and let $T:\Omega\to\Omega$ be continuous and positively homogeneous of degree one:
\begin{equation}\label{rl:eq:homogeneous}
  T(tx)=tT(x),
  \qquad t\ge0,\quad x\in\Omega.
\end{equation}
If $T^k(x)\to0$ for every $x\in\Omega$, then there exist $C\ge1$ and $\rho\in(0,1)$ such that
\begin{equation}\label{rl:eq:Rlinear-abstract}
  \norm{T^k(x)}\le C\rho^k\norm{x},
  \qquad x\in\Omega,\quad k\ge0.
\end{equation}
\end{theorem}

\begin{proof}
If $\Omega=\{0\}$, both conclusions are immediate.  Otherwise let
\[
  \Sigma:=\{x\in\Omega:\norm{x}=1\}.
\]
The set $\Sigma$ is nonempty and compact.  For each $x\in\Sigma$, pointwise convergence gives an integer $n_x\ge1$ such that $\norm{T^{n_x}(x)}<1/2$.  By continuity of $T^{n_x}$, there is a relative neighborhood $U_x$ of $x$ in $\Sigma$ on which
\[
  \norm{T^{n_x}(y)}\le\tfrac12.
\]
Choose a finite subcover $U_{x_1},\ldots,U_{x_r}$ and set
\[
  N_0:=\max_{1\le i\le r} n_{x_i}.
\]
Thus every $y\in\Sigma$ has an integer $m(y)\in\{1,\ldots,N_0\}$ for which $\norm{T^{m(y)}(y)}\le1/2$.  Continuity and compactness also give the finite transient bound
\begin{equation}\label{rl:eq:transient-bound}
  L:=\max_{0\le j\le N_0}\ \max_{y\in\Sigma}\norm{T^j(y)}<\infty.
\end{equation}

Fix $x\in\Omega\setminus\{0\}$.  Set $z_0=x$ and $s_0=0$.  Whenever $z_q\ne0$, choose
\[
  m_q:=m(z_q/\norm{z_q})\in\{1,\ldots,N_0\},
  \qquad
  s_{q+1}:=s_q+m_q,
  \qquad
  z_{q+1}:=T^{m_q}(z_q).
\]
Every iterate of $T$ is positively homogeneous, so
\[
  \norm{z_{q+1}}\le\tfrac12\norm{z_q}
  \quad\text{and hence}\quad
  \norm{z_q}\le2^{-q}\norm{x}
\]
for every constructed stage.  If $z_{q_*}=0$ for the first time, then $T(0)=0$ by \eqref{rl:eq:homogeneous}, and $T^k(x)=0$ for every $k\ge s_{q_*}$.  For $k<s_{q_*}$ there is a unique $q<q_*$ with $s_q\le k<s_{q+1}$.  If no zero occurs, then $s_q\ge q\to\infty$, so the same unique $q$ exists for every $k\ge0$.  In either case, whenever $T^k(x)$ has not already vanished, positive homogeneity and \eqref{rl:eq:transient-bound} give
\[
  \norm{T^k(x)}
  =\norm{T^{k-s_q}(z_q)}
  \le L\norm{z_q}
  \le L2^{-q}\norm{x}.
\]
Since $k<s_{q+1}\le(q+1)N_0$, one has $q>k/N_0-1$.  Therefore, also at the zero iterates,
\[
  \norm{T^k(x)}
  \le 2L\bigl(2^{-1/N_0}\bigr)^k\norm{x}.
\]
Equation \eqref{rl:eq:Rlinear-abstract} follows with $C=\max\{1,2L\}$ and $\rho=2^{-1/N_0}$.
\end{proof}

\section{Compatible endpoint cone and boundary continuity}\label{rl:sec:cone}

\subsection{Cone properties and the endpoint self-map}

\begin{lemma}\label{rl:lem:cone-closed}
The set $\Om$ is a closed cone.  If $(u,v)\in\Om$, then
\begin{equation}\label{rl:eq:support-inclusion}
  S(v)\subseteq S(u).
\end{equation}
\end{lemma}

\begin{proof}
Cone invariance follows from $v=q(H)u$.  Let $(u^\ell,v^\ell)\to(u,v)$ with $(u^\ell,v^\ell)\in\Om$.  Choose
\[
  v^\ell=q_\ell(H)u^\ell,
  \qquad
  q_\ell(t)=\prod_{j=1}^{p}(1-\alpha_j^\ell t),
\]
where the parameter vectors belong to the compact box $[\lambda_n^{-1},\lambda_1^{-1}]^p$.  A subsequence satisfies $\alpha_j^\ell\to\alpha_j$ for every $j$.  The corresponding polynomials converge at every spectral point to a polynomial $q\in\Pp$, and passage to the limit gives $v=q(H)u$.  Thus $(u,v)\in\Om$.  Finally, $v_i=q(\lambda_i)u_i$ gives \eqref{rl:eq:support-inclusion}.
\end{proof}

\begin{proposition}\label{rl:prop:selfmap-homogeneous}
The transition $T_p$ defined in \eqref{rl:eq:T-def} maps $\Om$ into itself.  For every $(u,v)\in\Om$ and every $t\ge0$,
\begin{equation}\label{rl:eq:F-homogeneous}
  F_p(tu,tv)=tF_p(u,v),
  \qquad
  T_p(tu,tv)=tT_p(u,v).
\end{equation}
\end{proposition}

\begin{proof}
Let $(u,v)\in\Om$.  If $s(u)>p$, Lemma~\ref{lem:enclosure} places every $\theta_j(u)$ in $[\lambda_1,\lambda_n]$.  Hence the reciprocal roots in \eqref{rl:eq:Psi} lie in $[\lambda_n^{-1},\lambda_1^{-1}]$, so $\Psi_u\in\Pp$ and
\[
  \bigl(v,F_p(u,v)\bigr)
  =\bigl(v,\Psi_u(H)v\bigr)\in\Om.
\]
If $s(u)\le p$, then \eqref{rl:eq:support-inclusion} gives $s(v)\le p$.  If $v=0$, any polynomial in $\Pp$ annihilates it.  If $v\ne0$, choose factors $1-t/\lambda_i$ for all $i\in S(v)$ and fill the remaining positions, if any, by repeating an arbitrary legal factor.  This produces a polynomial in $\Pp$ that annihilates $v$.  Since $F_p(u,v)=0$ on both finite-termination strata, $(v,0)\in\Om$.  Thus $T_p$ is a self-map.

For $t>0$, one has $S(tu)=S(u)$ and $K_p(tu)=tK_p(u)$.  On the regular stratum the two matrices in \eqref{rl:eq:K-pencil} are both multiplied by $t^2$, so the generalized eigenvalues and $\Psi_u$ are unchanged.  On the two finite-termination strata the output is zero.  Hence $F_p(tu,tv)=tF_p(u,v)$ for $t>0$.  At $t=0$, the state is $(0,0)$ and the same identity follows from the definition.  The identity for $T_p$ follows componentwise.
\end{proof}

\subsection{Cauchy--Binet coordinates}

If $p\ge n$, then $s(u)\le n\le p$ for every state, so $F_p\equiv0$ and continuity is immediate.  The formulas below are needed for possible regular states, which necessarily satisfy $n>p$.

Let $V\in\R^{n\times p}$ be the Vandermonde matrix
\begin{equation}\label{rl:eq:Vandermonde}
  V_{ij}=\lambda_i^{j-1},
  \qquad 1\le i\le n,\quad1\le j\le p.
\end{equation}
For a $p$-element set $I\subset\{1,\ldots,n\}$, let $V_I$ be the corresponding square row submatrix and set
\begin{equation}\label{rl:eq:Delta}
  \Delta_I:=\det(V_I)\ne0.
\end{equation}
Write $\Ip$ for all $p$-element index sets.  For $r\in\{1,\ldots,n\}$, define
\begin{equation}\label{rl:eq:beta-phi}
  \beta_I(u):=\Delta_I^2\prod_{i\in I}\lambda_i u_i^2,
  \qquad
  \phi_I^{(r)}:=\prod_{i\in I}\left(1-\frac{\lambda_r}{\lambda_i}\right).
\end{equation}

\begin{lemma}\label{rl:lem:CB-representation}
If $s(u)>p$, then for every spectral index $r$,
\begin{equation}\label{rl:eq:CB-ratio}
  \Psi_u(\lambda_r)
  =\frac{\displaystyle\sum_{I\in\Ip}\beta_I(u)\phi_I^{(r)}}
         {\displaystyle\sum_{I\in\Ip}\beta_I(u)}.
\end{equation}
The normalized coefficients are nonnegative and sum to one.
\end{lemma}

\begin{proof}
Let $S=S(u)$.  Since $|S|>p$, the active reduction indexed by $S$ is in the regular regime of Proposition~\ref{prop:det-formula}: the reduced Krylov matrix has full column rank.  Applying \eqref{eq:central-average} on that active subspace gives \eqref{rl:eq:CB-ratio} with both sums restricted to the $p$-element subsets of $S$.  If $I\not\subseteq S$, then $\beta_I(u)=0$, so extending the sums to all of $\Ip$ changes neither numerator nor denominator.  At least one $p$-element subset of $S$ exists, and every factor in its $\beta_I(u)$ is positive.  The denominator is therefore strictly positive.
\end{proof}

\subsection{Continuity across the finite-termination strata}

\begin{theorem}\label{rl:thm:joint-continuity}
The maps $F_p:\Om\to\R^n$ and $T_p:\Om\to\Om$ are continuous.  In particular, the transition is continuous at the full-rank finite-termination boundary $s(u)=p$ and at every rank-changing boundary $s(u)<p$.
\end{theorem}

\begin{proof}
Let $(u^\ell,v^\ell)\to(u,v)$ in $\Om$, and set $y^\ell:=F_p(u^\ell,v^\ell)$ and $y:=F_p(u,v)$.  Since the first component of $T_p$ is $v$, it is enough to prove $y^\ell\to y$.  Partition the indices into
\begin{equation}\label{rl:eq:index-partition}
  I_+:=\{\ell:s(u^\ell)>p\},
  \qquad
  I_-:=\{\ell:s(u^\ell)\le p\}.
\end{equation}
For $\ell\in I_-$, the definition gives $y^\ell=0$.  To prove convergence of the full sequence, it is enough in each case to prove that every infinite restricted subsequence indexed by $I_+$ or $I_-$ converges to the same limit $y$; a finite index class is irrelevant.

\emph{Case 1: $s(u)>p$.}
Choose $p+1$ nonzero coordinates of $u$.  They remain nonzero for all sufficiently large $\ell$, so eventually $\ell\in I_+$.  The matrix $K_p(u)$ has full column rank and
\[
  K_p(u)^THK_p(u)\succ0.
\]
The determinants in \eqref{rl:eq:Psi-det} depend continuously on $u$, and the limiting denominator is positive.  Hence, for every spectral coordinate $i$,
\begin{equation}\label{rl:eq:regular-Psi-cont}
  \Psi_{u^\ell}(\lambda_i)
  \longrightarrow \Psi_u(\lambda_i).
\end{equation}
It follows coordinatewise that $y^\ell\to\Psi_u(H)v=y$.

\emph{Case 2: $s(u)=p$.}
This is a full-column-rank finite-termination boundary, and $y=0$ by \eqref{rl:eq:F-def}.  The subsequence indexed by $I_-$ is identically zero.  On an infinite $I_+$ subsequence, the determinant argument from Case~1 remains valid because $K_p(u)$ still has rank $p$ and its limiting denominator is positive.  For $r\in S(u)$, \eqref{rl:eq:s-equals-p-annihilation} gives
\[
  \Psi_u(\lambda_r)=0,
\]
so \eqref{rl:eq:regular-Psi-cont} and $v_r^\ell\to v_r$ imply $y_r^\ell\to0$.  If $i\notin S(u)$, then $v_i=0$ by Lemma~\ref{rl:lem:cone-closed}.  Lemma~\ref{lem:enclosure} gives the uniform bound
\begin{equation}\label{rl:eq:uniform-multiplier-bound}
  \abs{\Psi_{u^\ell}(\lambda_i)}
  \le B_p,
  \qquad
  B_p:=\max\left\{1,\frac{\lambda_n}{\lambda_1}-1\right\}^{p},
\end{equation}
for $\ell\in I_+$, and therefore $y_i^\ell\to0$.  Both restricted subsequences converge to zero, hence so does the original sequence.

\emph{Case 3: $1\le s(u)<p$.}
Again $y=0$ and the $I_-$ subsequence is identically zero.  Consider an infinite $I_+$ subsequence.  For $i\notin S(u)$, \eqref{rl:eq:uniform-multiplier-bound} and $v_i^\ell\to0$ give $y_i^\ell\to0$.  Fix $r\in S(u)$.  We prove
\begin{equation}\label{rl:eq:Psi-to-zero}
  \Psi_{u^\ell}(\lambda_r)\longrightarrow0
  \qquad(\ell\in I_+).
\end{equation}
Every set $I\in\Ip$ containing $r$ has $\phi_I^{(r)}=0$.  Let $I\in\Ip$ omit $r$.  Since $|S(u)|<p$, the set $I$ contains an index $h\notin S(u)$.  Put
\begin{equation}\label{rl:eq:I-to-J}
  J:=(I\setminus\{h\})\cup\{r\}.
\end{equation}
For sufficiently large $\ell$, $u_r^\ell\ne0$.  If $\beta_I(u^\ell)=0$, its normalized contribution in \eqref{rl:eq:CB-ratio} is zero.  Otherwise every coordinate indexed by $I$ is nonzero, so $\beta_J(u^\ell)>0$, and
\begin{equation}\label{rl:eq:beta-ratio-boundary}
  \frac{\beta_I(u^\ell)}{\beta_J(u^\ell)}
  =\frac{\Delta_I^2}{\Delta_J^2}
   \frac{\lambda_h}{\lambda_r}
   \frac{(u_h^\ell)^2}{(u_r^\ell)^2}
  \longrightarrow0.
\end{equation}
If $\pi_I^\ell$ denotes the normalized coefficient in \eqref{rl:eq:CB-ratio}, then
\begin{equation}\label{rl:eq:pi-boundary}
  0\le\pi_I^\ell
  \le\frac{\beta_I(u^\ell)}{\beta_J(u^\ell)}
  \longrightarrow0.
\end{equation}
There are finitely many sets omitting $r$, and their templates are fixed finite real numbers.  The sets containing $r$ contribute zero, so \eqref{rl:eq:CB-ratio} proves \eqref{rl:eq:Psi-to-zero}.  Since $v_r^\ell\to v_r$, one has $y_r^\ell\to0$.  Thus the $I_+$ subsequence and the $I_-$ subsequence both converge to $y=0$.

\emph{Case 4: $u=0$.}
Compatibility gives $v=0$.  The $I_-$ subsequence again has zero output.  On $I_+$, \eqref{rl:eq:uniform-multiplier-bound} gives
\[
  \norm{y^\ell}\le B_p\norm{v^\ell}\longrightarrow0.
\]
Therefore $y^\ell\to0=y$.

The four cases cover every possible limit state.  Sequential continuity is equivalent to continuity in these finite-dimensional metric spaces.  The continuity of $T_p$ follows from that of $F_p$ and the continuous first component $(u,v)\mapsto v$.
\end{proof}

\section{Global \texorpdfstring{$R$}{R}-linear convergence}\label{rl:sec:rates}

\paragraph{Rate convention}
For fixed $H$ and $p$, let $z_k=T_p^k(z_0)$.  We say that the standard endpoint dynamics is state-uniformly $R$-linearly convergent if there exist $C\ge1$ and $\rho\in(0,1)$, independent of $z_0\in\Om$, such that
\[
  \norm{z_k}_\times
  \le C\rho^k\norm{z_0}_\times,
  \qquad k\ge0.
\]
Here $k$ counts completed sweeps.  Finitely terminating orbits are extended by zeros.

\subsection{Pointwise convergence input}\label{rl:sec:pointwise-input}

\begin{lemma}\label{rl:lem:realization}
Let $(u_0,u_1)\in\Om$.  There are stepsizes
\[
  \alpha_1,\ldots,\alpha_p\in[\lambda_n^{-1},\lambda_1^{-1}]
\]
such that
\begin{equation}\label{rl:eq:compatible-prefix}
  u_1=\prod_{j=1}^{p}(I-\alpha_jH)u_0.
\end{equation}
Use these steps as an admissible warm-up block for Algorithm~\ref{alg:standard-sweep}.  Until the algorithm terminates, let $G_j$ denote the history block starting at $u_j$ and ending at $u_{j+1}$.  Then
\begin{equation}\label{rl:eq:history-span-bridge}
  \Span(G_0)=\Kry_p(H,u_0),
\end{equation}
and the projected values extracted from $G_0$ agree exactly with those used in the definition of $F_p(u_0,u_1)$.  More generally, if
\[
  z_j:=(u_j,u_{j+1}),
\]
then the algorithmic endpoints and the abstract endpoints satisfy
\begin{equation}\label{rl:eq:orbit-correspondence}
  z_{j+1}=T_p(z_j)
\end{equation}
until finite termination, after which the abstract orbit reaches $(0,0)$ and remains there.  Consequently,
\begin{equation}\label{rl:eq:pointwise-endpoint}
  T_p^k(u_0,u_1)\longrightarrow(0,0)
  \qquad\text{for every }(u_0,u_1)\in\Om.
\end{equation}
\end{lemma}

\begin{proof}
Compatibility gives the representation
\eqref{rl:eq:compatible-prefix}; use these factors as an admissible
warm-up block. If an intermediate gradient vanishes, then $u_1=0$
and one endpoint transition reaches $(0,0)$.

Otherwise the history is completed and, by
Proposition~\ref{prop:krylov}, its span is $\Kry_p(H,u_0)$. If $s(u_0)>p$, the next endpoint is
$\Psi_{u_0}(H)u_1$. If $s(u_0)\le p$,
Proposition~\ref{prop:finite-branch} gives the zero output prescribed
by $F_p$. Repeating this argument shows that the algorithmic and
abstract endpoints agree until termination. The conclusion then
follows from Theorem~\ref{thm:master}.
\end{proof}

For an actual run with arbitrary positive finite warm-up steps, we initialize the endpoint estimate at
\[
  (u_0,u_1)=(g_{1,0},g_{2,0}),
\]
the start/end pair of the first completed standard sweep. If the warm-up steps themselves lie in $[\lambda_n^{-1},\lambda_1^{-1}]$, the warm-up start/end pair may be used as the initial state.

\subsection{The standard rank-compressed method}

\begin{theorem}\label{rl:thm:main-fixed}
Fix the reduced matrix $H$ in \eqref{eq:simple-spectrum} and an integer $p\ge1$.  Then there exist constants
\[
  C_* = C_*(H,p)\ge1,
  \qquad \rho_* = \rho_*(H,p)\in(0,1),
\]
such that for every initial state $z_0=(u_0,u_1)\in\Om$, if
\[
  z_k=T_p^k(z_0)=(u_k,u_{k+1}),
\]
then
\begin{equation}\label{rl:eq:endpoint-rate}
  \norm{(u_k,u_{k+1})}_\times
  \le C_*\rho_*^k\norm{(u_0,u_1)}_\times,
  \qquad k\ge0.
\end{equation}
Finite termination is included by extending the endpoint sequence with zeros.
\end{theorem}

\begin{proof}
Lemma~\ref{rl:lem:cone-closed} shows that $\Om$ is a finite-dimensional closed cone.  Proposition~\ref{rl:prop:selfmap-homogeneous} and Theorem~\ref{rl:thm:joint-continuity} show that $T_p$ is a continuous positively homogeneous self-map.  Lemma~\ref{rl:lem:realization} gives pointwise convergence for every state in $\Om$.  Theorem~\ref{rl:thm:homogeneous-rate} yields \eqref{rl:eq:endpoint-rate} with constants depending only on $(H,p)$.
\end{proof}

\subsection{Fixed spectral weights by conjugacy}

For a fixed positive spectral weight $W=\omega(H)$, let $T_{p,W}$ denote the endpoint transition obtained from \eqref{rl:eq:T-def} by using the weighted pencil \eqref{eq:weighted-pencil} on the regular stratum and the same zero output on the two finite-termination strata.  Set $S=W^{1/2}$ and define the product-space isomorphism
\begin{equation}\label{rl:eq:state-conjugacy-map}
  \mathcal S_W:\Om\to\Om,
  \qquad
  \mathcal S_W(u,v):=(Su,Sv).
\end{equation}

\begin{proposition}\label{rl:prop:endpoint-conjugacy}
The map $\mathcal S_W$ is a linear homeomorphism of $\Om$ onto itself, and
\begin{equation}\label{rl:eq:endpoint-conjugacy}
  \mathcal S_W T_{p,W}=T_p\mathcal S_W,
  \qquad\text{equivalently}\qquad
  T_{p,W}=\mathcal S_W^{-1}T_p\mathcal S_W.
\end{equation}
Thus the two maps have identical complete-sweep polynomials at corresponding conjugate states, conjugate orbit structure, identical Krylov ranks, and identical finite-termination indices.
\end{proposition}

\begin{proof}
If $v=q(H)u$ for some $q\in\Pp$, then $Sv=q(H)Su$ because $S$ commutes with $H$; the same argument with $S^{-1}$ shows that $\mathcal S_W$ maps $\Om$ onto itself.  Moreover,
\[
  K_p(Su)=SK_p(u),
\]
and hence
\[
  K_p(Su)^THK_p(Su)=K_p(u)^TWHK_p(u),
  \qquad
  K_p(Su)^TK_p(Su)=K_p(u)^TWK_p(u).
\]
Therefore the standard pencil at $Su$ is exactly the weighted pencil at $u$.  Since $S$ is invertible and acts by a positive scalar on each eigenspace of $H$, it preserves active support, Krylov rank, and both finite-termination strata.  On the regular stratum the two maps consequently use the same complete-sweep polynomial, while on the termination strata both outputs are zero.  This proves \eqref{rl:eq:endpoint-conjugacy}.
\end{proof}

\begin{corollary}\label{rl:cor:weighted-endpoint-rate}
Let $C_*$ and $\rho_*$ be the standard constants in Theorem~\ref{rl:thm:main-fixed}.  For every fixed positive spectral weight $W=\omega(H)$, every $z\in\Om$, and every $k\ge0$,
\begin{equation}\label{rl:eq:weighted-endpoint-rate}
  \norm{T_{p,W}^k(z)}_\times
  \le C_*\sqrt{\kappa_2(W)}\,\rho_*^k\norm{z}_\times.
\end{equation}
In particular, every fixed weight has the same admissible decay factor $\rho_* = \rho_*(H,p)$.  For $W=H^a$ with arbitrary $a\in\R$,
\begin{equation}\label{rl:eq:power-weight-prefactor}
  \sqrt{\kappa_2(W)}=\kappa_2(H)^{|a|/2}.
\end{equation}
\end{corollary}

\begin{proof}
Iterating \eqref{rl:eq:endpoint-conjugacy} and applying Theorem~\ref{rl:thm:main-fixed} give
\[
  \norm{T_{p,W}^k(z)}_\times
  \le \norm{S^{-1}}_2 C_*\rho_*^k\norm{\mathcal S_Wz}_\times
  \le C_*\norm{S^{-1}}_2\norm{S}_2\rho_*^k\norm{z}_\times.
\]
Since $S=W^{1/2}$, one has $\norm{S^{-1}}_2\norm{S}_2=\sqrt{\kappa_2(W)}$.  Equation \eqref{rl:eq:power-weight-prefactor} follows from the spectral condition number of $H^a$.
\end{proof}

\subsection{Inner iterates and optimization errors}

Let
\begin{equation}\label{rl:eq:MH}
  M_H:=\max_{\lambda,\theta\in[\lambda_1,\lambda_n]}
  \abs{1-\lambda/\theta}
  \le\max\left\{1,\frac{\lambda_n}{\lambda_1}-1\right\},
\end{equation}
and define
\begin{equation}\label{rl:eq:MHp}
  M_{H,p}:=\max_{0\le j\le p}M_H^j=\max\{1,M_H^p\}.
\end{equation}
For a nonterminating actual run, reindex the gradients from the start of the first sweep after the warm-up by
\begin{equation}\label{rl:eq:single-index}
  \widehat g_{kp+j}:=g_{k+1,j},
  \qquad k\ge0,\quad0\le j\le p-1.
\end{equation}
At a block boundary, \eqref{eq:block-indexing} gives
\[
  \widehat g_{(k+1)p}=g_{k+1,p}=g_{k+2,0},
\]
so the endpoint relation is
\begin{equation}\label{rl:eq:endpoint-single-index}
  u_k=\widehat g_{kp}=g_{k+1,0},
  \qquad k\ge0.
\end{equation}
The within-sweep bound, including the endpoint $j=p$, is
\begin{equation}\label{rl:eq:inner-bound}
  \norm{\widehat g_{kp+j}}\le M_H^j\norm{u_k},
  \qquad0\le j\le p.
\end{equation}
Any exact shorter block belongs to the finite-termination branch and does not occur on a nonterminating run.

\begin{corollary}\label{rl:cor:global-step-rate}
Fix $H$ and $p$, let $W=\omega(H)\succ0$, and consider the corresponding rank-compressed method with arbitrary positive finite warm-up stepsizes.  Define
\begin{equation}\label{rl:eq:weighted-inner-constants}
  \widehat\rho_*:=\rho_*^{1/p},
  \qquad
  \widehat C_W:=
  \max\left\{1,
  M_{H,p}C_*\sqrt{\kappa_2(W)}\,
  \widehat\rho_*^{-(p-1)}\right\},
\end{equation}
where $C_*$ and $\rho_*$ are the standard endpoint constants in Theorem~\ref{rl:thm:main-fixed}.  Either the method terminates finitely or the reindexed sequence \eqref{rl:eq:single-index} satisfies
\begin{equation}\label{rl:eq:global-gradient-rate}
  \norm{\widehat g_\ell}
  \le \widehat C_W\widehat\rho_*^\ell
  \norm{(u_0,u_1)}_\times,
  \qquad \ell\ge0,
\end{equation}
where $(u_0,u_1)=(g_{1,0},g_{2,0})$ is the start/end pair of the first completed sweep after the warm-up.  If $\widehat x_\ell$ denotes the correspondingly reindexed iterate, then
\begin{equation}\label{rl:eq:iterate-error-rate}
  \norm{\widehat x_\ell-x^\star}
  \le \lambda_1^{-1}\widehat C_W\widehat\rho_*^\ell
  \norm{(u_0,u_1)}_\times.
\end{equation}
The objective gap admits the bound
\begin{equation}\label{rl:eq:objective-gap-rate}
  0\le f(\widehat x_\ell)-f(x^\star)
  \le\frac{\widehat C_W^2}{2\lambda_1}
  (\widehat\rho_*^2)^\ell\norm{(u_0,u_1)}_\times^2.
\end{equation}
Thus every fixed spectral weight admits the same gradient and iterate decay factor $\widehat\rho_*$ and the same objective-gap factor $\widehat\rho_*^2$.
\end{corollary}

\begin{proof}
By Proposition~\ref{prop:spectral-conjugacy}, a nonterminating fixed-weight run corresponds to a nonterminating standard run, so every sweep has length $p$.  Write $\ell=kp+j$ with $0\le j\le p-1$.  Equations \eqref{rl:eq:inner-bound} and \eqref{rl:eq:weighted-endpoint-rate} give
\[
  \norm{\widehat g_\ell}
  \le M_{H,p}C_*\sqrt{\kappa_2(W)}\,
       \rho_*^k\norm{(u_0,u_1)}_\times.
\]
Since $k=(\ell-j)/p$,
\[
  \rho_*^k=\widehat\rho_*^{\ell-j}
  \le\widehat\rho_*^{-(p-1)}\widehat\rho_*^\ell.
\]
Equation \eqref{rl:eq:global-gradient-rate} follows from \eqref{rl:eq:weighted-inner-constants}.
The gradient-error equivalence
\begin{equation}\label{rl:eq:gradient-error-equivalence}
  \lambda_1\norm{x-x^\star}
  \le\norm{g(x)}
  \le\lambda_n\norm{x-x^\star}
\end{equation}
gives \eqref{rl:eq:iterate-error-rate}.  Finally,
\begin{equation}\label{rl:eq:objective-gap}
  0\le f(x)-f(x^\star)
  =\tfrac12g(x)^TH^{-1}g(x)
  \le\frac{1}{2\lambda_1}\norm{g(x)}^2,
\end{equation}
which gives \eqref{rl:eq:objective-gap-rate}.
\end{proof}

\begin{corollary}\label{rl:cor:repeated-eigenvalues}
All standard and fixed-weight endpoint, inner-gradient, iterate-error, and objective-gap estimates in this section remain valid when $H$ has repeated eigenvalues.  The common decay factors $\rho_*$ and $\widehat\rho_*$ may be chosen to depend only on $p$ and the distinct eigenvalues of $H$.
\end{corollary}

\begin{proof}
Apply the isometric spectral-projector reduction described in Sections~\ref{sec:setup} and~\ref{rl:sec:sweeps}. It preserves compatible endpoint pairs, projected pencils, all polynomial endpoint and inner-sweep updates, and their Euclidean norms. The reduced endpoint cone contains every active coordinate pattern, so the reduced constants apply uniformly to all initial active subsets. The iterate-error and objective-gap estimates then follow from the same spectral bounds.
\end{proof}

\section{Special cases and relation to earlier analyses}\label{sec:scope}

\begin{corollary}\label{cor:standard-harmonic}
The global-convergence and $R$-linear results apply to
\begin{enumerate}[label=\textnormal{(\roman*)},leftmargin=2.2em]
\item standard Ritz LMSD, $W=I$;
\item harmonic-Ritz LMSD, $W=H$;
\item every fixed power weight $W=H^a$, $a\in\R$;
\item the delayed BB1 and BB2 recurrences obtained for $p=1$.
\end{enumerate}
All four cases admit the standard endpoint decay factor $\rho_*(H,p)$ and the standard per-step factor $\widehat\rho_* = \rho_*^{1/p}$.  The endpoint prefactor is $C_*$ for $W=I$, at most $C_*\sqrt{\kappa_2(H)}$ for $W=H$, and at most $C_*\kappa_2(H)^{|a|/2}$ for $W=H^a$.
\end{corollary}

\begin{proof}
These choices are positive spectral weights, and for $p=1$ the corresponding reciprocal generalized Rayleigh quotients are the quadratic BB1 and BB2 stepsizes.  Apply Proposition~\ref{rl:prop:endpoint-conjugacy}, Corollary~\ref{rl:cor:weighted-endpoint-rate}, and \eqref{rl:eq:power-weight-prefactor}.
\end{proof}

Curtis and Guo \cite{CurtisGuo2018} obtain finite-cycle contractions for standard and harmonic-Ritz LMSD under Assumption~3.4. Their constants inherit the normalized-conditioning parameter through equation~(3.12c) and Lemmas~3.10--3.11. Theorem~\ref{rl:thm:main-fixed} instead uses the compatible-state formulation and algebraic rank compression, and Corollary~\ref{rl:cor:weighted-endpoint-rate} transfers the resulting rate to fixed positive spectral weights by conjugacy.

For $W=I$, Theorem~\ref{thm:master} also recovers Fletcher's global-convergence result \cite{Fletcher2012} through the complete-sweep determinant representation.

\newcommand{\NumericalRankTolerance}{10^{-12}}
\newcommand{\NumericalStoppingTolerance}{10^{-11}}

\section{Numerical diagnostics}\label{sec:numerics}

All computations were carried out in double precision with MATLAB R2023a. Ritz values were applied in nonincreasing order. Power weights were formed in the logarithmic domain and normalized by a common positive scale. All experiment uses the known diagonal spectral representation.

A history matrix $G$ was assigned numerical rank equal to the number of singular values satisfying
\begin{equation}\label{eq:numerical-rank-rule}
  \sigma_j>
  \max\bigl\{
    \NumericalRankTolerance\,\sigma_1,\,
    \eps_{\rm mach}\max(\operatorname{size}(G))\sigma_1
  \bigr\}.
\end{equation}
The retained left singular vectors formed the basis of the compressed projected pencil.  Iterative runs stopped when
\[
  \frac{\norm{g}}{\norm{g_0}}
  \le\NumericalStoppingTolerance,
\]
and Cauchy steps were used for the warm-up block.

For reproducibility, define the deterministic coefficients
\begin{equation}\label{eq:deterministic-gradient}
  d_i(s)
  :=1+\sin\bigl(0.37(i+s)\bigr)
    +0.3\cos\bigl(1.17(i+s)\bigr)
    +0.5\sin\bigl(0.11(i+s)\bigr).
\end{equation}
The initial gradients used below are specified in terms of these coefficients.

The experiment combines the power-weight comparison, repeated eigenvalues, and actual-history conditioning. We take $p=5$ and twenty distinct eigenvalues, consisting of four equally spaced points in each of
\[
 [1,2],\quad [10,15],\quad [50,60],\quad [200,250],\quad [800,1000].
\]
Every distinct eigenvalue has multiplicity three, so the unreduced dimension is $N=60$.  In the diagonal spectral coordinates, the initial gradient is
\[
  (g_0)_i=d_i(7),
  \qquad i=1,\ldots,60,
\]
and therefore has a nonzero projection onto every eigenspace.  We compare
\[
  a\in\{-1,-\tfrac12,0,\tfrac12,1,\tfrac32,2\},
  \qquad W=H^a.
\]

Since the minimizer has been translated to the origin, the objective gap is evaluated in spectral coordinates as
\[
  f(x_j)-f_\star
  =\frac12g_j^TH^{-1}g_j.
\]
The quantity displayed in the left panel is the relative gap
\[
  r_j
  :=\frac{f(x_j)-f_\star}{f(x_0)-f_\star}.
\]
The left panel of Figure~\ref{fig:weight-objective-conditioning} shows the raw
relative objective gap after every gradient step.  The right panel shows
\[
  \chi_k=\frac{\norm{g_{k,0}}}{\sigma_{\min}(G_k)}
\]
for the standard and harmonic choices, which are the two variants considered in the earlier conditioning-based analysis.

\begin{figure}[!htbp]
  \centering
  \IfFileExists{fig_weight_objective_conditioning.pdf}{%
    \includegraphics[width=\textwidth]{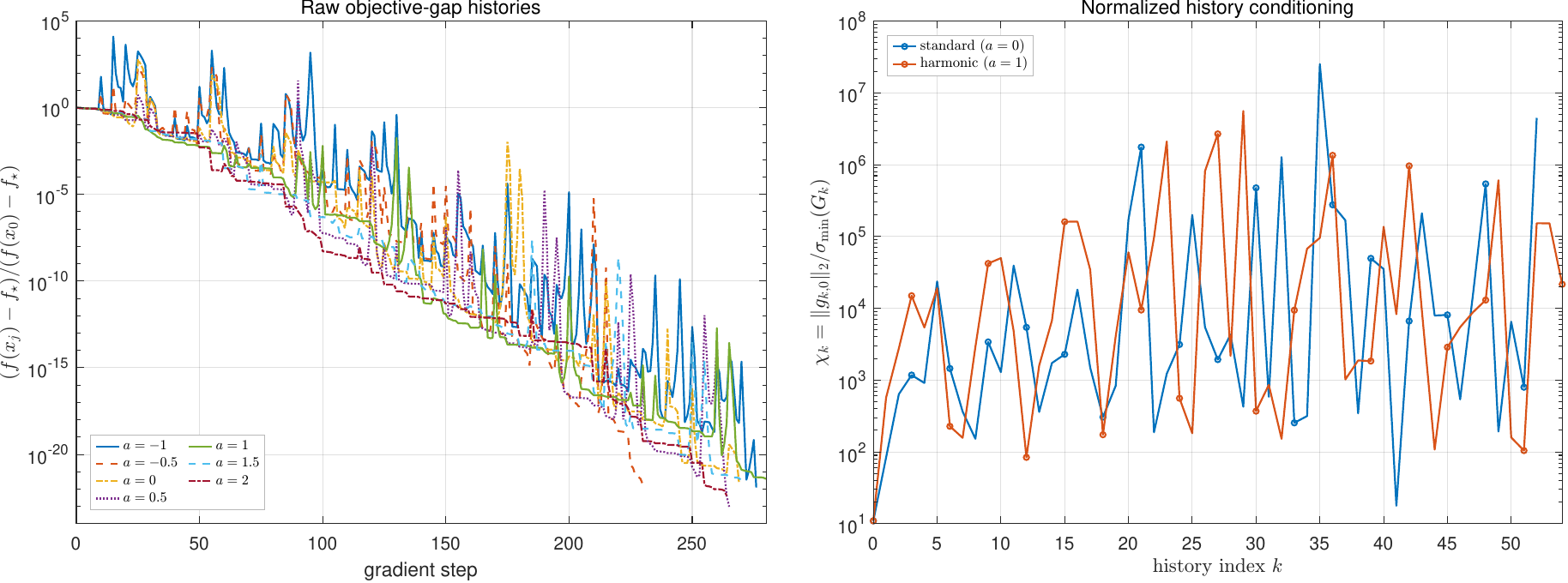}%
  }{%
    \fbox{\parbox[c][4.0cm][c]{0.94\textwidth}{\centering
      Run \texttt{run\_weighted\_lmsd\_revised\_experiments.m} to generate
      \texttt{fig\_weight\_objective\_conditioning.pdf}.}}%
  }
  \caption{A repeated-spectrum quadratic with $N=60$, twenty distinct eigenvalues, and $p=5$.}
  \label{fig:weight-objective-conditioning}
\end{figure}

\begin{table}[!htbp]
\centering
\caption{Diagnostics for the repeated-spectrum quadratic.  $A_k=\max_{0\le j\le p}\norm{g_{k,j}}/\norm{g_{k,0}}$.}
\label{tab:weight-convergence}
\small
\setlength{\tabcolsep}{4pt}
\begin{tabular}{@{}crrrr@{}}
\toprule
$a$ & \shortstack{gradient\\steps} & \shortstack{weighted\\sweeps} &
$\max_k\chi_k$ & $\max_k A_k$\\
\midrule
$-1$   & 276 & 55 & $1.30\times10^{8}$ & $1.33\times10^{3}$\\
$-0.5$ & 231 & 46 & $2.35\times10^{8}$ & $1.55\times10^{5}$\\
$0$    & 269 & 53 & $2.51\times10^{7}$ & $2.26\times10^{5}$\\
$0.5$  & 265 & 52 & $1.59\times10^{7}$ & $6.93\times10^{4}$\\
$1$    & 280 & 55 & $5.59\times10^{6}$ & $3.30\times10^{4}$\\
$1.5$  & 270 & 53 & $1.99\times10^{7}$ & $2.21\times10^{4}$\\
$2$    & 264 & 52 & $1.48\times10^{7}$ & $6.99\times10^{1}$\\
\bottomrule
\end{tabular}
\end{table}

All seven weights reach the prescribed gradient tolerance in 231--280
gradient steps.  The raw objective histories are strongly nonmonotone and can
undergo substantial transient growth, but the overall decay on the
logarithmic scale is consistent with the global $R$-linear result.

All detected histories have rank five, while the normalized quantities attain
\[
  \max_k\chi_k=2.51\times10^7
  \quad\hbox{and}\quad
  5.59\times10^6,
\]
for the standard and harmonic runs, respectively. Similar, and in some cases substantially larger, values were already reported by Curtis and Guo \cite[Table~1]{CurtisGuo2018}. The observations illustrate that full rank alone provides neither a moderate nor an a priori bound on the normalized conditioning constant.

\FloatBarrier

\section{Conclusions and limitations}
\label{sec:combined-conclusion}

Exact algebraic rank compression gives a common formulation of the regular and finite-termination strata of the standard complete-sweep method. If a block-start gradient has at most $p$ active distinct eigenvalues, the following delayed sweep terminates finitely. On nonterminating trajectories, the determinant and Cauchy--Binet representation of the complete-sweep multiplier yields pointwise global convergence across changes in Krylov rank.

Consecutive endpoints define a continuous positively homogeneous self-map on the compatible-state cone. Pointwise convergence and compactness give an endpoint $R$-linear estimate with constants $C_*(H,p)$ and $\rho_*(H,p)$ uniform over compatible initial states. Endpoint conjugacy transfers the same admissible decay factor to every fixed positive spectral weight, with Euclidean-norm prefactor at most $C_*\sqrt{\kappa_2(W)}$. The estimates extend to inner gradients, iterate errors, objective gaps, and matrices with repeated eigenvalues.

The analysis concerns strictly convex quadratics in exact arithmetic with algebraic rank compression and fixed positive spectral weights. Stability under numerical rank thresholds, trajectory-dependent weights, and nonquadratic extensions remains open.

\section*{Acknowledgments}
Shutai Yang thanks Shixiang Chen for supervising his undergraduate thesis at the University of Science and Technology of China, and Xiaowei Xu for supervising his project under the National College Students Innovation and Entrepreneurship Training Program at the same university. The present article grew out of those two undergraduate projects. He also thanks Gexuan Zhu and Bowen Deng for carefully reading an earlier Chinese-language manuscript presenting the global convergence part of this work and for offering helpful comments and suggestions.

\paragraph{Use of generative AI}
During the preparation of this manuscript, Shutai Yang used OpenAI language models for language and \LaTeX{} editing, the presentation of mathematical notation, and assistance in revising the MATLAB scripts used for the numerical experiments.  He independently checked and revised all model-assisted material and made all scholarly judgments and final decisions.  Both authors reviewed and approved the final manuscript.  The authors assume responsibility for all content.

\bibliographystyle{siamplain}
\bibliography{LMSD_refs_20260814}
\end{document}